\documentclass{article}

\usepackage[english]{babel}

\usepackage[letterpaper,top=2cm,bottom=2cm,left=3cm,right=3cm,marginparwidth=1.75cm]{geometry}

\usepackage{amsmath}
\usepackage{amsthm}
\usepackage{graphicx}
\usepackage[colorlinks=true, allcolors=blue]{hyperref}
\usepackage{amssymb}
\usepackage{tikz}
\usetikzlibrary{positioning}
\usepackage[utf8]{inputenc}
\usepackage[T1]{fontenc}
\usepackage{tcolorbox}
\usepackage{enumitem}
\usepackage{multicol}
\usepackage[table]{xcolor} 
\usepackage{array} 
\usetikzlibrary{arrows}
\usetikzlibrary{arrows.meta}

\usepackage{subcaption} 

\usetikzlibrary{external}
\usepackage{svg}

\newtheorem{theorem}{Theorem}[section]
\newtheorem{corollary}[theorem]{Corollary}

\newtheorem{lemma}[theorem]{Lemma}
\newtheorem{problem}[theorem]{Problem}

\newtheorem{observation}[theorem]{Observation}

\usepackage{xcolor}
\usepackage{ifthen}

\newboolean{shownotes}
\setboolean{shownotes}{true} 

\newcommand{\note}[1]{%
  \ifthenelse{\boolean{shownotes}}%
    {\textcolor{red}{[Note: #1]}}%
    {}%
}

\newcommand{\R}[0]{\mathbb{R}}

\newcommand{\Nb}[0]{\mathbb{N}}

\newcommand{\Zb}[0]{\mathbb{Z}}

\newcommand{\Ic}[0]{\mathcal{I}}

\newcommand{\Rc}[0]{\mathcal{R}}

\newcommand{\Fc}[0]{\mathcal{F}}
\newcommand{\Ec}[0]{\mathcal{E}}
\newcommand{\Ac}[0]{\mathcal{A}}

\DeclareMathOperator{\lcm}{lcm}

\title{Coloring Questions on Axis-Parallel Rectangles and Arithmetic Progressions}
\author{Gábor Damásdi}

\begin{document}
\maketitle

\begin{abstract}
We present an explicit family of hypergraphs with arbitrarily large uniformity and chromatic number that admit realizations in both geometric and number-theoretic settings. As an application,  we give a new proof of a theorem of Chen, Pach, Szegedy, and Tardos. They showed that for any constants $c,k\ge1$, there exists a finite point set $P$ in the plane with the following property: for every coloring of $P$ with $c$ colors, there is an axis-parallel rectangle containing at least $k$ points, all of the same color. Their original proof is probabilistic; we present an explicit construction. Moreover, in the case $k=2$, we show that one can even realize a graph that has arbitrarily large girth and chromatic number simultaneously. 
We also answer a question of Pálvölgyi on coloring sets of integers with respect to certain finite arithmetic progressions. Finally, we give an application to coloring partially ordered sets.


\end{abstract}

\section{Introduction}
Coloring questions on graphs and hypergraphs arising from geometric or number-theoretic configurations play an important role in combinatorics. In this paper we present a family of hypergraphs of large chromatic number that appear naturally in both  scenarios. 

A proper $c$-coloring of a hypergraph $H=(V,E)$ is a map $V\to [c]$ such that no edge is monochromatic. The \emph{chromatic number} of $H$ is the smallest $c$ for which such a coloring exists. 

Given a set $P$ of points in the plane and a family $\Fc$ of planar sets, let $H(P,\Fc)$ denote the hypergraph given by incidence, i.e., the vertex set is $P$ and the edge set is $\{P\cap F\mid  F\in \Fc\}$. If the members of $\Fc$ belong to a  geometric family (e.g., disks, rectangles, convex sets), then we say that $H(P,\Fc)$ is \emph{realized} by that family. 
Coloring questions on geometric hypergraphs have a vast literature, see \cite{damasdi2025coloring} for a recent survey and the webpage \cite{cogezoo} for up-to-date results. Some of the families that will be of interest to us are:
\begin{itemize}[topsep=5pt, partopsep=0pt, itemsep=0pt, parsep=0pt]
    \item Axis-parallel rectangles: sets of the form $\{(x,y)\mid x_1\le x\le x_2,y_1\le y\le y_2\}$ for some $x_1,x_2,y_1,y_2\in \R$.
    \item Bottomless axis-parallel rectangles: sets of the form $\{(x,y)\mid x_1\le x\le x_2,y\le y_2\}$ for some $x_1,x_2,y_2\in \R$.
    \item Horizontal strips: sets of the form $\{(x,y)\mid y_1\le y\le y_2\}$ for some $y_1,y_2\in \R$.
\end{itemize}


 In this paper, we focus on the case where $\Fc$ consists of axis-parallel rectangles. This was first considered by Chen, Pach, Szegedy, and Tardos \cite{chen2009delaunay}. Using a probabilistic argument, they showed that we can realize hypergraphs of arbitrarily large chromatic number and uniformity.

\begin{theorem}[Chen, Pach, Szegedy, and Tardos \cite{chen2009delaunay}]\label{thm:chenpach}
    For any constants $c,k\ge1$, there exists a finite set $P$ of points in the plane with the following property: for every coloring of $P$ with $c$ colors, there is an axis-parallel rectangle containing $k$ points, all of the same color.
\end{theorem}

 For any $k\ge 1$, Chekan and Ueckerdt gave an explicit set $P$ of points and a family $\Fc$ of horizontal strips and bottomless axis-parallel rectangles such that $H(P,\Fc)$ is $k$-uniform and not 2-colorable \cite{chekan2022polychromatic}. Since strips
and bottomless rectangles can be replaced by ordinary  rectangles without changing the incidence structure, their construction yields an explicit solution to Theorem \ref{thm:chenpach} in the case $c=2$. We generalize their construction to all $c>2$ and prove a slightly stronger statement. A set of points is called \emph{ascending} if the $x$-coordinate order of the points is the same as the $y$-coordinate order of the points. We say that a family of intervals $\Ic$ is \emph{nested} if for any $A,B\in \Ic$ we have $A\subseteq B$, $B\subseteq A$ or $A\cap B= \emptyset$.

 \begin{theorem}\label{thm:ascending}
     For any constants $c,k\ge1$, there exists an explicit finite set $P$ of points and a family $\Rc$ of axis-parallel rectangles in the plane with the following property: the $y$-projections of the rectangles in $\Rc$ form a \textbf{nested} family, and for every coloring of $P$ with $c$ colors, there is a rectangle in $\Rc$ that contains $k$ \textbf{ascending} points of $P$, all of the same color.
 \end{theorem}

Finding graphs and hypergraphs that have large girth and chromatic number is a classical topic, going back to Tutte \cite{D54}  (under the pseudonym Blanche Descartes) and Erdős \cite{E59}. Large-girth constructions were recently found in some geometric problems; see, for example, \cite{davies2020box,davies2024solution}. Using a graph construction from \cite{kostochka1999properties} we show a large-girth variant of Theorem \ref{thm:chenpach} in the graph case $k=2$.

 \begin{theorem}\label{thm:girtthm}
     For any constants $c,g\ge1$, there exists a finite set $P$ of points and a set $\Rc$ of axis-parallel rectangles in the plane such that $H(P,\Rc)$ is a graph, has girth at least $g$ and chromatic number at least $c+1$.
 \end{theorem}

\subsection{Arithmetic progressions and a question of Pálvölgyi}

One of the most famous coloring results in combinatorial number theory is van der Waerden’s theorem  \cite{vanderWaerden1927}. It says that for any constants $c,k\ge1$ there is an $n$ such that if we color the integers from 1 to $n$ using $c$ colors, then we always find a monochromatic arithmetic progression of length $k$. 

Pálvölgyi proposed the following variant (see Problem 9.11 in  \cite{damasdi2025coloring}).
For $D\subset \Nb$, let $\Ac_D$ denote the family of all finite arithmetic progressions whose difference is in $D$. Instead of all arithmetic progressions, we will only consider arithmetic progressions from $\Ac_D$. On the other hand, we allow for restrictions on the base set. Formally, for a base set $V\subset \Zb$ and a family $\Ac$ of finite arithmetic progressions, let $H(V,\Ac)$ denote the incidence hypergraph on vertex set $V$ whose edge set is $\{A\cap V \mid A\in \Ac\}$. If the members of $\Ac$ belong to $\Ac_D$, then we say that $H(V,\Ac)$ is \emph{realized} by arithmetic progressions whose difference is in $D$.

This setting is clearly analogous to the problem of axis-parallel rectangles. Instead of a planar set of points, we have a set of integers. Instead of axis-parallel rectangles, we have arithmetic progressions (by pure coincidence, both can be abbreviated as AP). Again we are interested in the coloring properties of these hypergraphs. 

\begin{problem}[Pálvölgyi, Problem 9.11 in  \cite{damasdi2025coloring}]
    For which $D\subset \Nb$ and $k\ge1$ can we realize $k$-uniform hypergraphs of large chromatic number using arithmetic progressions whose difference is in $D$?
\end{problem}

Pálvölgyi noted that for finite $D$ the chromatic number is bounded and depends only on the divisibility lattice of $D$. On the other hand, if $D$ is infinite, almost nothing is known; see \cite{bursics} for partial results.
Pálvölgyi also pointed out that the case $D=\{2^i|i\in \Nb \}$ shows a connection to geometric hypergraphs, as we can realize all finite hypergraphs that are realizable by points and bottomless rectangles in the plane (see the discussion at \cite{domotorpMO451443}). We show that the connection is much stronger. In fact,  this case is equivalent to realizing hypergraphs using axis-parallel rectangles whose $y$-projection is nested.

\begin{theorem}\label{thm:equiv}
A hypergraph is realizable by a family of axis-parallel rectangles whose $y$-projections form a nested family if and only if it is realizable by arithmetic progressions whose difference is a power of 2.    
\end{theorem}

The connection relies on a mapping that uses the van der Corput sequence. We also extend one direction to arbitrary infinite difference sets.

\begin{theorem}\label{thm:general}
Suppose $D\subset \Nb$ is an infinite set. If a hypergraph is realizable by a family of axis-parallel rectangles whose $y$-projection is nested, then it is realizable by arithmetic progressions whose difference is in $D$.    
\end{theorem}

Theorem \ref{thm:ascending} and Theorem \ref{thm:general} immediately imply the following resolution of Pálvölgyi's problem.

\begin{theorem}
    For any infinite $D$ and any constants $c,k\ge1$ there exists a set $V\subset \Zb$ such that for any $c$-coloring of $V$ there is a finite arithmetic progression $A$ whose difference is in $D$ such that $A\cap V$ is monochromatic and has size at least $k$.
\end{theorem}

\subsection{Coloring Hasse diagrams of posets}
For a partially ordered set $(P,<)$ its Hasse diagram is the graph on vertex set $P$ where $u<v$ is an edge if there is no $w\in P$ such that $u<w<v$. As Hasse diagrams are triangle-free, it is interesting to consider their coloring properties.  Erdős and Hajnal constructed posets whose Hasse diagrams have large chromatic number \cite{erdHos1964some}. Recently Suk and Tomon \cite{suk2021hasse} constructed Hasse diagrams with $n$ vertices and chromatic number $\Omega(n^{1/4})$.

 K{\v{r}}{\'\i}{\v{z}} and Ne{\v{s}}et{\v{r}}il showed that there are 2-dimensional posets such that their Hasse diagram has large chromatic number \cite{kvrivz1991chromatic}. A standard way to produce posets from a set of planar points is to consider the ordering $p < q \Leftrightarrow x(p)<x(q)$ and $y(p)<y(q)$. This gives a 2-dimensional poset, and every $2$-dimensional poset can be represented this way. Note that $\{p,q\}$ is an edge of the Hasse diagram of $(P,<)$ if and only if they form an ascending set and there is an axis-parallel rectangle $R$ such that $P\cap R=\{p,q\}$. Therefore, Theorem \ref{thm:ascending} implies the following strengthening of the result of K{\v{r}}{\'\i}{\v{z}} and Ne{\v{s}}et{\v{r}}il.

\begin{corollary}
    For any constants $k,c\ge 1$, there exists a 2-dimensional poset $P$ with the following property: for every coloring of $P$ with $c$ colors, there is a monochromatic increasing path in the Hasse diagram of $P$ of length $k$.
\end{corollary}

In Section \ref{sec:constr} we present a construction of a family of large-chromatic hypergraphs. In Section \ref{sec:georel} we show how to realize these using rectangles (Theorem \ref{thm:ascending}). Finally, in Section \ref{sec:arithmetic} we show how to reduce Pálvölgyi's question to the problem on axis-parallel rectangles (Theorems \ref{thm:equiv} and \ref{thm:general}).

\section{Constructing the hypergraph}\label{sec:constr}
In this section we define, for each $c,k\ge 1$ a $k$-uniform hypergraph $H_{k}^c$ such that every $c$-coloring of $H_{k}^c$ yields a monochromatic edge. As a warm-up, we recall a simpler construction, the $k$-ary tree hypergraph.

Suppose $T=(V,E)$ is a rooted tree. Let $H_T$ denote the following hypergraph on vertex set $V$. For each leaf, the vertices of the unique root-leaf path form an edge. We will call these the \emph{path edges}. For each non-leaf vertex $v$, the children of $v$ form an edge, these are the \emph{sibling edges}. 

We say that a rooted tree is of \emph{depth $k$} if each root-leaf path contains exactly $k$ vertices. If $T$ is the tree of depth $k$ where each non-leaf vertex has $k$ children, then $H_T$ is a $k$-uniform hypergraph, known as the \emph{$k$-ary tree hypergraph}. It is easy to see that $H_T$ is not 2-colorable for any $T$. Indeed, suppose that we have a 2-coloring of the vertices. We can either follow the root's color down a path edge, or we get stuck at a vertex whose children form a monochromatic sibling edge.

The $k$-ary tree hypergraphs played an important role in a number of problems on geometric hypergraphs, see \cite{abab,DP20,PTT09} and Section 2.2 in \cite{damasdi2025coloring}. Unfortunately, if $k$ is large enough, then they cannot be realized by axis-parallel rectangles \cite{kumar2017minimum}. The path edges are realizable on their own, but we cannot add the sibling edges. To remedy this problem, we modify the construction. The idea is to take many trees instead of just one. We realize the path edges in each tree and we will replace the sibling edges with a new type of edge called \emph{transversal edges}. Each transversal edge takes at most one vertex from each tree, and using them we will show that we can follow the color of the root down a path edge in at least one of our trees. The precise construction is as follows.

For each $c,k\ge 1$, we inductively construct a vertex-ordered $k$-uniform hypergraph $H_k^c$ of chromatic number at least $c+1$. As mentioned, the construction is a generalization of a construction by Chekan and Ueckerdt for 2-colors \cite{chekan2022polychromatic}, and we mostly follow their terminology. For $c=1$ we take $k$ vertices in arbitrary order and a single edge containing all of them. 

Let $k\ge 1$ and $c>1$ be fixed. By induction, $H_{k}^{c-1}$ exists. Let $m$ denote the number of its vertices. We start by building an auxiliary rooted forest consisting of $m^k$ rooted trees. The vertex set of $H_k^c$ will be partitioned into \emph{stages}, each stage contains at most one vertex from each tree. The vertices of a given stage will be at a fixed distance $j$ from the roots of their corresponding tree, that we will call the \emph{level} of the stage. Each stage $S$ of level $j$  has exactly $m^{k-j}$ vertices and  comes with a fixed linear ordering $<_S$ on these vertices.

\begin{figure}[!ht]
    \centering
    \includegraphics[width=0.7\linewidth]{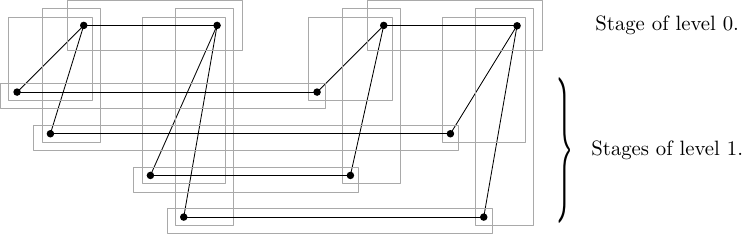}
    \caption{The hypergraph $H^2_{2}$ and its realization by axis-parallel rectangles.}
    \label{fig:h22}
\end{figure}

Furthermore, we regard the vertices of a stage as partitioned into $m^{k-j-1}$ disjoint \emph{blocks} of $m$ consecutive vertices in the ordering $<_S$. We will consider subsets of a stage that contain exactly one vertex from each block, hence the following definition will be useful. For an ordered set $A=\{a_1,a_2,\dots ,a_{tm}\}$  let $f_m(A)$ denote the family of subsets of $A$ that contain exactly one element from each block $\{a_{im+1},a_{im+2},\dots, a_{im+m}\}$,  $i=0,\dots, t-1$. 

We start building the forest by taking $m^k$ ordered vertices, they form the unique stage of level 0. They also serve as the roots of the $m^k$ trees. Suppose that we have already defined a stage $S$ of level $j<k-1$. Then for each subset $S'\in f_m(S)$ we define a new stage $T(S')$ of level $j+1$ on $m^{k-j-1}$ new vertices. Each vertex of $S'$ receives exactly one child in $T(S')$ and the vertices of $T(S')$ inherit the ordering of their corresponding parents. For stages of level $k-1$ the process stops. Hence, we obtain a forest of $m^k$ trees, whose vertices are partitioned into stages of level $0,1,\dots, k-1$. Let $V$ denote the vertex set of the forest; this will be the vertex set of $H_k^c$. For a vertex $v\in V$, let $root(v)$ denote the root of the tree that contains $v$. Let $path(v)$ denote the vertex set of the unique path from $v$ to $root(v)$.

Now we are ready to define the edges of the hypergraph $H_k^c$. There are two types of edges.  Let $\Ec_P=\{path(v)\mid \text{$v$ appears in a stage of level $k-1$}\}$, that is, for each vertex $v$ that appears in a stage of level $k-1$, add $path(v)$ as an edge to $H_k^c$. As before, these are called \emph{path edges}. 

Secondly, for each stage $S$ of level $j$ do the following. Recall that the vertices of $S$ are partitioned into $m^{k-j-1}$ disjoint blocks of $m$ consecutive elements in the ordering $<_S$. In each of these blocks, add edges to form a copy of the vertex-ordered hypergraph $H_k^{c-1}$ such that the vertex ordering of $H_k^{c-1}$ matches the restriction of $<_S$ to that block. That is, we realize $m^{k-j-1}$ disjoint copies of $H_k^{c-1}$ in the stage. The edges appearing in these will be called \emph{transversal} hyperedges, and the set of all transversal hyperedges appearing in any stage is denoted by $\Ec_T$. 

Finally we set $H_k^c=(V,\Ec_P\cup \Ec_T)$, see Figure \ref{fig:h22} for an example.

\begin{lemma}
    The $k$-uniform hypergraph $H_k^c=(V,\Ec_P\cup \Ec_T)$ is not properly $c$-colorable. 
\end{lemma}
\begin{proof}
    We use induction on $c$. For $c=1$, we have a single edge; hence the hypergraph is not 1-colorable. For $c>1$ suppose that there is a proper $c$-coloring and fix a color, say red. We claim that each stage $S$ of level $j$ contains at least $m^{k-j-1}$ red vertices, one from each block of $m$ points. In other words, there is a red set in $f_m(S)$. Indeed, the transversal hyperedges form $m^{k-j-1}$ disjoint copies of $H_k^{c-1}$. By induction, each color must appear in each of these copies, hence we find a red vertex in each. 

    \begin{figure}[!ht]
        \centering
        \includegraphics[width=0.99\linewidth]{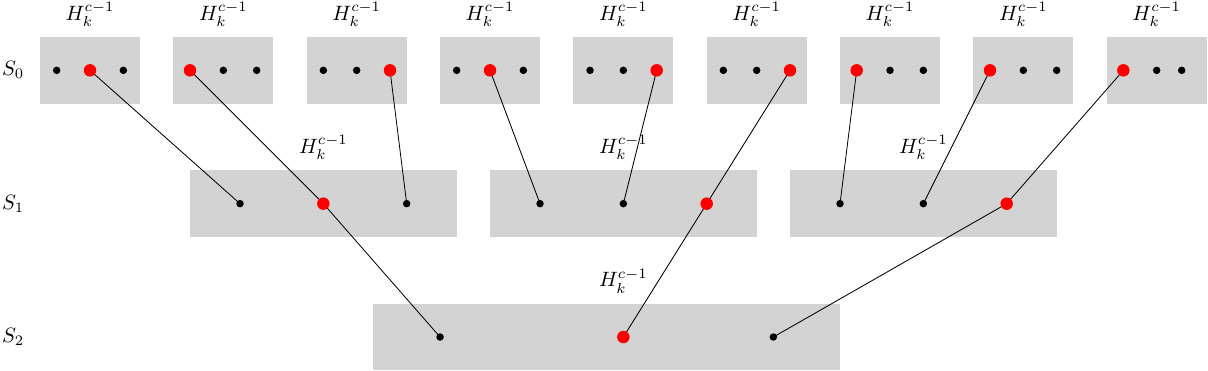}
        \caption{Finding a monochromatic path edge in the $k=3$ case. Grey rectangles indicate a copy of $H_k^{c-1}$, red vertices indicate the elements in $B_j$ in each stage.}
        \label{fig:structure}
    \end{figure}
    
    Using this observation and induction on $j$, we show that for each $j\in \{0,\dots, k-1\}$ there is a stage $S_j$ of level $j$ and a subset $B_j\in f_m(S_j)$ such that for each $v\in B_j$ the entire path $path(v)$ is red. See Figure \ref{fig:structure} for an example. For $j=0$, let $S_0$ be the unique stage at level 0. It consists of $m^{k-1}$ blocks, each forming a copy of $H_k^{c-1}$, we pick a red vertex from each to obtain $B_0$. Assume that we have found $B_j$ for some $j<k-1$. As $B_j\in f_m(S_j)$ there is a child stage $S_{j+1}=T(B_j)$. By the induction hypothesis, we can pick a red vertex in each block of $S_{j+1}$, giving us $B_{j+1}$. Each vertex of $B_{j+1}$ has its parent in $B_j$, so $path(v)$ is red for all $v\in B_{j+1}$.

    Finally, we find $B_{k-1}$ inside a stage of level $k-1$. It contains $m^{k-(k-1)-1}=m^0$ vertices, i.e., a single vertex $v$. Then $path(v)$ is monochromatic (red), a contradiction. Hence, there is no proper $c$-coloring of the hypergraph $H_k^c$. 
    
\end{proof}

\section{Geometric realization}\label{sec:georel}
We prove Theorem \ref{thm:ascending} in two steps. First, we show that there is a realization of $H_k^c$ by axis-parallel rectangles that contain ascending sets. Then we show how to achieve nested $y$-projection.

 For a point $p\in\R^2$, write $x(p),y(p)$ for its coordinates. A set $P=\{p_1,\dots,p_n\}$ ordered by $x$-coordinate  is \emph{ascending} if $y(p_1)<\dots<y(p_n)$ and \emph{descending} if $y(p_1)>y(p_2)>\dots>y(p_n)$. We say that \emph{$p$ is above $q$} if $y(p)> y(q)$ and use analogous definitions for \emph{below, to the right, and to the left}.

\begin{lemma}\label{lem:realize}
    For every $c,k\ge 1$, the hypergraph $H_k^c$ can be realized by axis-parallel rectangles such that each rectangle contains an ascending set of $k$ points. 
\end{lemma}
\begin{proof}
    We use induction on $c$. For $c=1$ the hypergraph $H_k^c$ has a single edge, and hence we can pick any set of $k$ ascending points. For $c>1$, we find the realization of $H_k^c$ in two phases. In the first phase, we embed the vertices such that the following properties hold.
    \begin{itemize}
        \item[a)] The vertices of a given stage $S$ lie on a  horizontal line, the \emph{stage line of $S$}, and are ordered according to $<_S$ along the line.
        \item[b)] The path edges are realizable by axis-parallel rectangles in this embedding.
    \end{itemize}

    Then, in the second phase, we perturb the vertices in each stage so that the transversal hyperedges also become realizable.

    \textbf{Phase 1}. We start by placing the roots (the level-0 stage) on a horizontal line in order. 
    Then iteratively pick a stage $S$ that is already embedded but whose child stages $T_1,T_2,\dots, T_r$ are not, and embed the child stages simultaneously in the following way. Pick $r$ distinct horizontal lines below the line of $S$ and above the next stage line, if there is one.  For each $i\in \{1,\dots, r\}$ place the vertices of $T_i$ on the $i$-th line so that  each vertex has the same $x$-coordinate as its parent in $S$. (This also ensures that $T_i$ is ordered according to $<_{T_i}$ along the stage line.)
    
       Next, for each vertex $v\in S$ slightly shift the children of $v$ to the left in the following way. Let $w$ be the vertex preceding $v$ in the $x$-coordinate ordering of the whole set of points,  if there is any. We shift the children of $v$ so that each stays on its stage line, they form a descending set, and they stay between $v$ and $w$ in the $x$-coordinate ordering, see Figure \ref{fig:shift}.

      \begin{figure}[!ht]
          \centering
          \includegraphics[width=0.7\linewidth]{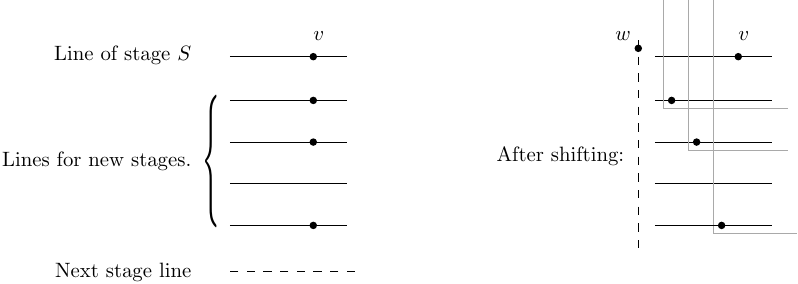}
          \caption{The children of $v$ are placed below $v$ then shifted to the left to form a descending set.}
          \label{fig:shift}
      \end{figure}
    
    After this process, each non-root vertex is below and to the left of its parent, and hence each $path(v)$ is an ascending set. Furthermore, when we shift the vertices to the left, the order of the vertices within a stage does not change.  

    For each vertex $v$, let $R_v$ be the rectangle whose bottom-left corner is $v$ and whose top-right corner is $root(v)$. The rectangle $R_v$ contains $path(v)$, since the path is ascending. We claim that it contains no other vertex.
    
     We only need to show that any vertex $w\notin path(v)$ does not lie in $R_v$. As $v$ lies between $root(v)$ and the preceding root vertex, we may assume that $root(v)=root(w)$. Let $p$ denote the highest level vertex in $path(v)\cap path(w)$.  Let $v'$ and $w'$ denote the two children of $p$ in $path(v)$ and $path(w)$, respectively. By the definition of $p$, we know that $v'$ and $w'$ are distinct vertices and that they were created when we processed the stage containing $p$. By construction, the children of $p$ are descending, hence  $v'$ is either above and to the left of $w'$ or below and to the right of $w'$, see Figure \ref{fig:pathedges}. 

    \begin{figure}[!ht]
        \centering
        \includegraphics[width=0.4\linewidth]{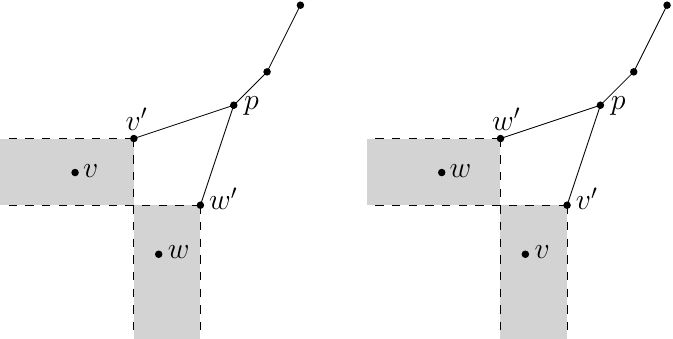}
        \caption{The two cases depending on how $v'$ and $w'$ were placed. The gray regions indicate the possible positions of $v'$ and $w'$ in the plane.}
        \label{fig:pathedges}
    \end{figure}
    Suppose $v'$ is above and to the left of $w'$. All descendants of $w'$, including $w$, are below $w'$. On the other hand,  all descendants of $v'$ are placed later than $w'$, hence they are on  lines above the line of $w'$.  Therefore, $v$ is above $w'$, and $w$ cannot be in $R_v$. 
    
    Suppose $v'$ is below and to the right of $w'$. All descendants of $w'$, including $w$, are to the left of $w'$. As the descendants of $v'$ are placed later than $w'$, they are all to the right of it. Hence $v$ is to the right of $w$ and $w$ cannot be in $R_v$.  Thus $R_v$ contains exactly the path vertices, and every path edge is realized.

    \textbf{Phase 2.}    Next, we realize the transversal edges. Right now the vertices of a given stage $S$ of level $j$ lie on a horizontal line and their $x$-coordinate order matches $<_S$, that is,  the $m^{k-j-1}$ copies of $H_k^{c-1}$ that are formed by the transversal edges occupy disjoint consecutive blocks of $m$ vertices. 

    Note that if we have a realization of any hypergraph by axis-parallel rectangles, then one can obtain new realizations by changing the position of the points. As long as the $x$ and $y$-coordinate orders of the points do not change, we can realize the same hypergraph using rectangles.
    
      For each block, replace the $m$ points by a very thin realization of $H_k^{c-1}$ whose vertices are close to the corresponding vertices in the block. If the copy is thin enough, the rectangles cannot contain points from any other stage. If the perturbation is small enough,  the existing path edges do not change. Hence, we have realized $H_k^{c}$ by axis-parallel rectangles. Each realizing rectangle contains an ascending
$k$-set by construction in Phase 1 and by induction in Phase 2.
    
\end{proof}

Next, we show that there is a realization where the $y$-projections of the rectangles form a nested family. Recall that a family of intervals $\Ic$ is \emph{nested} if for any $A,B\in \Ic$ we have $A\subseteq B$, $B\subseteq A$ or $A\cap B= \emptyset$. (We allow the family to contain an interval more than once.) 

\begin{lemma}\label{lem:nested}
    There is a realization of $H_k^c$ with axis-parallel rectangles such that the projection of the rectangles to the $y$-axis is nested. 
\end{lemma}
\begin{proof}
    We modify the construction in Lemma \ref{lem:realize}. Again we proceed by induction on $c$, the $c=1$ case is trivial.

When realizing path edges via rectangles $R_v$,  the top side of each rectangle lies on the level-0 stage line. Hence, all path rectangles share the same upper $y$-endpoint in projection, hence their $y$-projections form a nested family. Enlarge  each slightly so that no point lies on any rectangle boundary.

For transversal edges, each is realized inside an arbitrarily thin horizontal neighborhood of a stage line. If this neighborhood is thin enough, then the $y$-projection of any transversal rectangle is either fully contained in the $y$-projection of a given path rectangle or disjoint from it (because points are not on boundaries). Projections from different stages are disjoint by construction.

Finally, within a single stage, we may place the disjoint copies of $H_k^{c-1}$ at slightly different vertical offsets (still within the stage neighborhood), so that the projection intervals of rectangles from different copies are disjoint. Inside each copy, nestedness holds by induction. Therefore, the $y$-projections form a nested family.
\end{proof}

\subsection{Large girth variant}

For $g\ge2$, a $g$-cycle of a hypergraph is an alternating sequence, $(v_1, E_1, v_2, E_2,\dots,v_g, E_g)$, of distinct vertices $v_1,\dots, v_g$ and distinct edges $E_1,\dots, E_g$ such that $v_g, v_1\in E_g$, and $v_i, v_{i+1}\in E_i$, for $i = 1,\dots,g - 1$. The \emph{girth} of a hypergraph is the length of the smallest cycle in the hypergraph. 

We show that in the graph case $k=2$ we can achieve arbitrarily large girth. Let us recall a construction from \cite{kostochka1999properties} using our terminology of stages. For each $c,g\ge 1$ we construct an ordered graph $G^c(g)$ of girth at least $g$ and chromatic number at least $c+1$. We use induction on $c$. For $c=1$, the graph $K_2$ works.

For $c>1$, fix an  auxiliary hypergraph $H=(V_H,\Ec_H)$ of uniformity $|G^{c-1}(g)|$, girth $g$, and chromatic number at least $c+1$ (Such hypergraphs exist by standard probabilistic arguments, see \cite{ErdosHajnal1966}). We use $H$ to define our stages. 

 Create a stage of level 0 that contains the vertices of $H$, but, contrary to the previous construction, do not place any edges inside this stage. All of the remaining stages are going to be of level 1. For each edge $S\in \Ec_H$ add a stage $T(S)$ of level 1. It contains a child vertex $v'$ for each $v\in S$. We connect $v$ and $v'$, and $<_{T(S)}$ is the ordering inherited from $S$. As $H$ is $|G^{c-1}(g)|$-uniform, $T(S)$ contains $|G^{c-1}(g)|$ vertices. Add edges in $T(S)$ to realize a single copy of the graph $G^{c-1}(g)$ such that the ordering of the vertices matches the ordering $<_{T(S)}$.  

\begin{lemma}
    The chromatic number of $G^{c}(g)$ is at least $c+1$.
\end{lemma}
\begin{proof}
    We proceed by induction on $c$. For $c=1$  the chromatic number is 2. For $c\ge2$ suppose $G^{c}(g)$ admits a  $c$-coloring.  Since $H$ has a chromatic number of at least $c+1$, in any $c$-coloring of the stage of level 0 some edge $S\in \Ec_H$ is monochromatic. This implies that on stage $T(S)$ one of the colors is forbidden. But $T(S)$ spans a copy of $G^{c-1}(g)$, whose chromatic number is at least $c$ by induction, a contradiction.
\end{proof}

\begin{lemma}\label{lem:girth}
    The girth of $G^{c}(g)$ is at least $g$.
\end{lemma}
\begin{proof}
    For any fixed $g$, we proceed by induction on $c$. For $c=1$ we have chosen $K_2$, a graph that contains no cycles. For $c>1$ suppose that $C=(v_1,\dots, v_t)$ is a cycle in $G^{c}(g)$.
    
      If each $v_i$ is in the same stage of level 1, then $t\ge g$ by induction. Otherwise, let $x_1,x_2,\dots, x_m$ be the vertices of $C$ that lie in the level-0 stage, listed in cyclic order along $C$. There are no edges in the stage of level 0, nor between the stages of level 1, so the part of $C$ between $x_i$ and $x_{i+1}$ is non-empty and must run in a single stage of level 1 corresponding to some hyperedge $E_i\in \Ec_{H}$. We claim that there is a cycle of length at most $m$ in $H$. Indeed, consider $(x_1,E_1,x_2,E_2,\dots, x_m,E_m)$. By construction $x_m, x_1\in E_m$, and $x_i, x_{i+1}\in E_i$, for $i = 1,\dots,m - 1$. The $x_i$ vertices are distinct. If the $E_i$'s are also distinct, then we have found a cycle of $H$. If the $E_i$ are not distinct, then fix a pair $i, j$ where $E_i=E_j$, $i\ne j$  and $|i-j|$ is minimal. Then $(x_{i+1},E_{i+1},x_{i+2},E_{i+2},\dots, x_{j},E_{j})$ is a cycle of $H$ of length at most $m$.  Since $H$ has girth at least $g$, we obtain $t\ge m \ge g$, that is, the length of $C$ is at least $g$.
    
\end{proof}

The realization of $G^{c}(g)$ by axis-parallel rectangles is essentially the same as in Lemma \ref{lem:realize} (place stages on separated
horizontal lines and realize each stage-internal copy), we omit the details. This finishes the proof of Theorem \ref{thm:girtthm}.

\section{Arithmetic progressions}\label{sec:arithmetic}

In this section we turn our attention to coloring integers with respect to arithmetic progressions. After some preparation, we show that realizability by nested-projection rectangles is equivalent to realizability
by finite arithmetic progressions whose difference is a power of 2.

We will use the following notation for the closed horizontal strip and the upward open horizontal strip: $$S{[a,b]}=\{(x,y)\mid a\le y\le b\} \qquad  S{[a,b)}=\{(x,y)\mid a\le y< b\}.$$

The \emph{van der Corput} sequence is a well-studied notion in discrepancy theory, see for example \cite{matousek1999geometric}. It is defined as follows. Suppose the binary expansion of $n\in \Nb$ is 

\[
n=\sum_{i=0}^{\ell} d_i 2^i,\qquad d_i\in\{0,1\},
\]
then the $n$th term is
\[
a_n=\sum_{i=0}^{\ell} d_i 2^{-i-1}.
\]

We will show that subsequences of the van der Corput sequence provide an interface between the two problems. The connection to coloring points with respect to rectangles is not new. See for example \cite{pach2010coloring}, and the following problem that was proposed by Gábor Tardos for the 2015 Schweitzer competition.
\begin{problem}[Miklós Schweitzer Memorial Competition, 2015. Problem 2.]
 
Let $\{a_n\}$ be the van der Corput sequence. Let $P$ be the set of points on the plane that have the form $(n, a_n)$. Let $G$ be the graph with vertex set $V$ that is connecting any two points $(p, q)$ if there is an axis-parallel rectangle $R$ such that $R\cap V = \{p, q\}$. Prove that the chromatic number of $G$ is finite.
\end{problem}

For a nested family $\Ic$, we can define a rooted forest $T_\Ic$ based on $\Ic$ in the following way. The vertices of $T_\Ic$ are the elements of $\Ic$, and we connect $A$ to $B$ if $A\subseteq B$ but there is no $C\in \Ic$ such that $A\subsetneq C\subsetneq B$. The roots are the maximal elements. Equivalently, $T_\Ic$ is the Hasse diagram of the poset $(\Ic,\subseteq)$. We say that a nested family is a \emph{perfect nested family of depth $t$} if $T_\Ic$ is a perfect binary tree with  $t$ levels so that  each non-leaf interval is the union of its two children. Equivalently, a nested family $\Ic$ is a perfect nested family of depth $t$ if we can index its elements using the $0/1$-sequences of length at most $t-1$ such that $I_{s_1s_2\dots s_r}=I_{s_1s_2\dots s_r0}\cup I_{s_1s_2\dots s_r1}$ for all $s_1,\dots,s_r\in \{0,1\}$ with $r<t-1$.
 
For example, for any $t\in \Nb$, the interval family $\{[\frac{i}{2^k},\frac{i+1}{2^k} )\}$ over all $0\le k<t$ and $0\le i < 2^k$ is a perfect nested family of depth $t$. Note that this is the same family as $\{[a_b, a_b + 2^{-k}) \}$ over all $0\le k<t$ and $0\le b < 2^k$. 
    \begin{observation}\label{obs:extend}
    Any nested family can be extended to a perfect nested family of depth $t$ for some sufficiently large $t\in \Nb$.    
    \end{observation}

First, we focus on the $D=\{2^i\mid i\in \Nb\}$ case. Let $\Ac_{2^i}$ denote all finite arithmetic progressions with difference from $\{2^i\mid i\in \Nb\}$. This case is essentially equivalent to realizing hypergraphs using nested axis-parallel rectangles.

\begin{proof}[Proof of Theorem \ref{thm:equiv}]

\emph{(Arithmetic progressions $\Rightarrow$ nested rectangles)}.
First, suppose  $V\subset \Zb$. We want to show that $H(V,\Ac_{2^i})$ is realizable by axis-parallel rectangles whose $y$-projections are nested. Consider the planar set of points $P_V=\{(n,a_n)\mid n\in V \}$, where $a_i$ denotes the $i$-th term of the van der Corput sequence. We claim that  every edge of $H(V,\Ac_{2^i})$ can be realized by an axis-parallel rectangle and the resulting family has nested $y$-projections.

Consider the \textbf{infinite} arithmetic progression $A=\{2^tk+b\mid k\in \Zb \}$ for some fixed $t\in \Nb$ and $b<2^t$.  For any $n\in \Nb$, we have $$a_n=\sum\limits_{i=0}^{\ell} d_i2^{-i-1}=\sum\limits_{i=0}^{t-1} d_i2^{-i-1}+\sum\limits_{i=t}^{\ell} d_i2^{-i-1}=a_{(n \bmod 2^t)}+\sum\limits_{i=t}^{\ell} d_i2^{-i-1}.$$ 

Moreover, the values $\{a_0,\dots,a_{2^t-1}\}$ are spaced by at least $2^{-t}$. Therefore,
\[
n\equiv b\pmod{2^t}\quad\Longleftrightarrow\quad a_b \le a_n < a_b+2^{-t}.
\]

Thus the points $\{(n,a_n)\mid n\in A\cap V\}$ are exactly the points of $P_V$ that lie in the strip $S[a_b,a_b+2^{-t})$. Choosing a finite consecutive subset of $A$ corresponds to intersecting
this strip with a vertical slab in $x$, i.e., to an axis-parallel rectangle. It is also easy to see that the family of intervals $\{[a_b,a_b+2^{-t}) \mid b,t\in \Nb \text{ and } b<2^t] \} $ is nested, see Figure \ref{fig:vandermap}. Hence, we can represent each edge using an upward open rectangle, and we can choose them such that their $y$-projection is nested. 

\begin{figure}[!ht]
    \centering
    \includegraphics[width=0.7\linewidth]{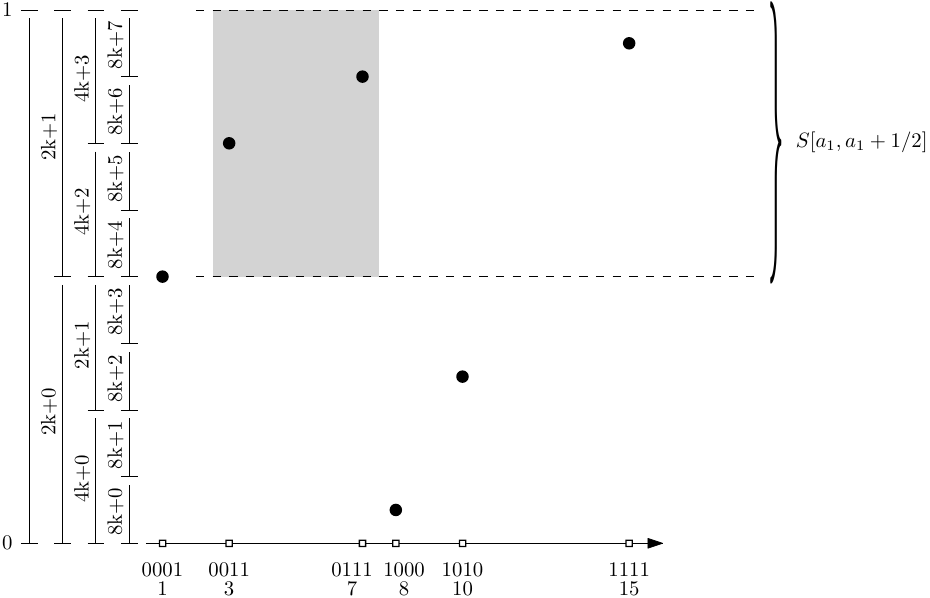}
    \caption{The point set given by the van der Corput sequence for $V=\{1,3,7,8,10,15\}$. The grey rectangle captures $A\cap V$ for the arithmetic progression $A=3,5,7$.}
    \label{fig:vandermap}
\end{figure}

\emph{(Nested rectangles $\Rightarrow$ arithmetic progressions )}. Let $P=\{p_1,\dots,p_n\}\subset\R^2$ be ordered by increasing $x$-coordinate, and let $\Rc$ be a family of axis-parallel rectangles whose $y$-projections form a nested family $\Ic$. 
By Observation~\ref{obs:extend}, extend $\Ic$ to a perfect nested family $\Ic'$ of depth $t$. Index $\Ic'$ by binary strings of length at most $t-1$ as $\{I_{s_1\cdots s_r}\}$, such  that $I_{s_1s_2\dots s_r}=I_{s_1s_2\dots s_r0}\cup I_{s_1s_2\dots s_r1}$ for all $s_1,\dots,s_r\in \{0,1\}$ with $r<t-1$.

For each point $p_i=(x_i,y_i)$, since $\Ic'$ partitions its top interval into leaf intervals, there is a unique leaf $I_{s_1\cdots s_{t-1}}$ containing $y_i$. Define
\[
v_i=\Bigl(\sum_{j=1}^{t-1} s_j 2^{j-1}\Bigr) + i\,2^{t-1},
\qquad
V=\{v_1,\dots,v_n\}\subset\Nb.
\]
Because the first term is $<2^{t-1}$, we have $v_1<\cdots<v_n$, so the order on $V$ matches the $x$-order on $P$.

Now consider a strip corresponding to an interval $I_{q_1\cdots q_r}$ (with $r\le t-1$). Let
\[
A=\Bigl\{ k2^{r} + \sum_{j=1}^{r} q_j2^{j-1}\ \Bigm|\ k\in\Zb\Bigr\}.
\]
This is an arithmetic progression with difference $2^r$. Consider a point $p_i=(x_i,y_i)$ and let $I_{s_1\dots s_{t-1}}$ denote the leaf  interval that contains $y_i$. Since $\Ic'$ is perfect, $p_i$ is in the strip  $S[I_{q_1q_2\dots q_r}]$ if and only if $q_1\dots q_r$ is a prefix of $s_1\dots s_{t-1}$. That is, if and only if  $v_i$ is of the form  $$v_i=\bigl(\sum\limits_{j=1}^{r} q_j2^{j-1}\bigr)+\bigl(\sum\limits_{j=r+1}^{t-1} s_j2^{j-1}\bigr)+i2^{t-1}=\bigl(\sum\limits_{j=1}^{r} q_j2^{j-1}\bigr)+2^{r}\bigl(\bigl(\sum\limits_{j=r+1}^{t-1} s_j2^{j-1-r}\bigr)+i2^{t-1-r}\bigr).$$ Hence, this prefix condition is equivalent to $v_i\equiv \sum_{j=1}^{r} q_j2^{j-1}\pmod{2^r}$, i.e., to $v_i\in A$. Hence intersections of strips (and therefore rectangles) with $P$ correspond to intersections of power-of-two progressions with $V$.

\end{proof}

\subsection{The general case: Proof of Theorem \ref{thm:general}}

\begin{proof}[Proof of Theorem \ref{thm:general}]
  Let $D\subset\Nb$ be infinite. We will use a rapidly growing subsequence of $D$.

Define $d_1<d_2<\cdots$ recursively as follows:
\[
d_{i}=\min\Bigl\{ d\in D : d>2^{i-1}\cdot \lcm(d_1,\dots,d_{i-1})\Bigr\}.
\]
Then for each $i\ge1$ we have
\begin{equation}\label{eq:growth}
\lcm(d_1,\dots,d_{i+1})\ \ge\ d_{i+1}\ >\ 2^i\,\lcm(d_1,\dots,d_i).
\end{equation}

 This ensures that certain modular equations are solvable. The following lemma captures this property.  

\begin{lemma}\label{lem:enough}
Let $d_1<d_2<\cdots$ satisfy~\eqref{eq:growth}. If the modular system $\{x\equiv r_i \pmod{d_i}\}_{i=1}^j$ has a solution, then there are at least $2^j$ distinct choices of $r_{j+1}\in\{0,1,\dots,d_{j+1}-1\}$ such that $\{x\equiv r_i \pmod{d_i}\}_{i=1}^{j+1}$ is solvable.
\end{lemma}

\begin{proof}
Let $L=\lcm(d_1,\dots,d_j)$ and suppose the system has a solution, i.e.,\ a residue class $x\equiv r\pmod{L}$. Consider
\[
\ell_i=r+iL,\qquad i=0,1,\dots,2^j-1.
\]
Each $\ell_i$ solves the first $j$ congruences. Choose $r_{j+1}\equiv \ell_i\pmod{d_{j+1}}$; then $\ell_i$ is a solution of the extended system.

We claim the residues $\ell_i\bmod d_{j+1}$ are all distinct for $0\le i<2^j$. If $\ell_{i_1}\equiv \ell_{i_2}\pmod{d_{j+1}}$ for $i_1\ne i_2$, then $d_{j+1}\mid (\ell_{i_1}-\ell_{i_2})$. But $L\mid (\ell_{i_1}-\ell_{i_2})=(i_1-i_2)L$ as well, so $\lcm(L,d_{j+1})=\lcm(d_1,\dots,d_{j+1})$ divides the difference. On the other hand,
\[0<
|\ell_{i_1}-\ell_{i_2}| \le (2^j-1)L < 2^j L < \lcm(d_1,\dots,d_{j+1}),
\]
where the strict inequality uses~\eqref{eq:growth}. This contradiction shows distinctness, giving at least $2^j$ choices.
\end{proof}

 Let $P=\{p_1,\dots,p_n\}\subset\R^2$ be ordered by increasing $x$-coordinate, and let $\Rc$ be a family of axis-parallel rectangles whose $y$-projections form a nested family $\Ic$. By Observation~\ref{obs:extend}, extend $\Ic$ to a perfect nested family $\Ic'$ of depth $t$. Index $\Ic'$ by binary strings of length at most $t-1$ as $\{I_{s_1\cdots s_r}\}$, such  that $I_{s_1s_2\dots s_r}=I_{s_1s_2\dots s_r0}\cup I_{s_1s_2\dots s_r1}$ for all $s_1,\dots,s_r\in \{0,1\}$ with $r<t-1$.

For each interval $I_{s_1\cdots s_m}$ at depth $m$ we will choose a residue
\[
r_{s_1\cdots s_m}\in\{0,1,\dots,d_m-1\}
\]
such that:
\begin{enumerate}[label=(\roman*),leftmargin=2.2em]
\item for every string $s_1\cdots s_m$, the system $\{x\equiv r_{s_1\cdots s_j}\pmod{d_j}\}_{j=1}^{m}$ is solvable;
\item for each fixed $m$, the $2^m$ values $\{r_{s_1\cdots s_m}\}$ are pairwise distinct.
\end{enumerate}
We choose these residues inductively over $m$. For the step $m\to m+1$, each solvable system at depth $m$ has at least $2^m$ extensions by Lemma~\ref{lem:enough}; we choose distinct residues for the children to satisfy (ii).

For each leaf string $s_1\cdots s_{t-1}$, fix one integer $f_{s_1\cdots s_{t-1}}<\lcm(d_1,\dots, d_{t-1})$
solving $\{x\equiv r_{s_1\cdots s_j}\pmod{d_j}\}_{j=1}^{t-1}$.  For each $p_i=(x_i,y_i)$, let $I_{s_1\cdots s_{t-1}}$ be the unique leaf interval containing $y_i$ and set
\[
v_i = f_{s_1\cdots s_{t-1}} + i\cdot \lcm(d_1,\dots,d_{t-1}),
\qquad
V=\{v_1,\dots,v_n\}\subset\Zb.
\]
Since $f_{s_1\cdots s_{t-1}}<\lcm(d_1,\dots, d_{t-1})$, we have $v_1<\cdots<v_n$. Hence the order on $V$ matches the $x$-order on $P$. 

Finally, consider any strip corresponding to $I_{q_1\cdots q_r}$ (with $r\le t-1$). Let $A$ be the set of integers satisfying the congruence 
\[
x\equiv r_{q_1\cdots q_r}\pmod{d_r}
.\]
This set is an arithmetic progression with common difference $d_r\in D$. Consider a point $p_i=(x_i,y_i)$ and let $I_{s_1\dots s_{t-1}}$ denote the leaf interval that contains $y_i$. Since $\Ic'$ is perfect, $p_i$ is in the strip  $S[I_{q_1q_2\dots q_r}]$ if and only if $q_1\dots q_r$ is a prefix of $s_1\dots s_{t-1}$. The integer $v_i$ is a solution of $\{x\equiv r_{s_1\cdots s_j}\pmod{d_j}\}_{j=1}^{t-1}$. Hence, if $q_1\dots q_r$ is a prefix of $s_1\dots s_{t-1}$, then $v_i$ is a solution of $x\equiv r_{q_1\cdots q_r}\pmod{d_r}
$. On the other hand, if it is not a prefix, then $v_i$ is a solution of $x\equiv r_{s_1\cdots s_r}\pmod{d_r}
$, and by $(ii)$ we have $r_{s_1\cdots s_r}\neq r_{q_1\cdots q_r}$. Hence, $p_i$ is in the strip $S[I_{q_1q_2\dots q_r}]$ if and only if $v_i$ is in the arithmetic progression $A$. Therefore strip (and hence rectangle) intersections correspond to intersections with arithmetic progressions of allowed differences, completing the realization.

    \end{proof}

\section{Final remarks}
As we have seen, for $k=2$ we have constructions of large girth. The main reason why the same ideas do not work for $k=3$ is that certain pairs of path edges of a tree overlap in more than one vertex.    
\begin{problem}
    Is it true that for any constants $c,k,g\ge1$ there  exists a finite set of points and a finite set of axis-parallel rectangles such that their incidence hypergraph is  $k$-uniform, has chromatic number at least $c$ and girth at least $g$?
\end{problem}

We note that even the following easier question is open. 

\begin{problem}
    Is it true that for any $c\ge1$ there exists a finite set of points and a finite set of axis-parallel rectangles such that their incidence hypergraph is $3$-uniform, has chromatic number at least $c$ and any two edges intersect in at most one vertex?
\end{problem}

\section*{Acknowledgments}
This project was initiated at the  Early Career Researchers in Combinatorics 2024 workshop in Edinburgh. The author thanks the organizers for creating a stimulating research environment. The author is grateful to Tim Planken for pointing out the connection to the construction of Chekan and Ueckerdt. I also thank Dömötör Pálvölgyi for suggesting that there is a connection between the two problems and for reviewing the manuscript.

\bibliographystyle{abbrv}
\bibliography{sample}

\end{document}